\documentclass[12pt]{amsart}

\usepackage[T1]{fontenc}
\usepackage[utf8]{inputenc}

\usepackage{amsmath,amssymb,amsthm}
\usepackage[backend=biber,style=numeric,doi=true,url=true,isbn=false]{biblatex}
\usepackage[colorlinks,linkcolor=blue,citecolor=blue,urlcolor=blue,pdfpagelabels,pdffitwindow=false]{hyperref}

\usepackage{mathrsfs}
\DeclareMathAlphabet{\mathcal}{OMS}{cmsy}{m}{n}

\newtheorem{theorem}{Theorem}[section]

\newtheorem{lemma}[theorem]{Lemma}
\theoremstyle{remark}
\newtheorem{remark}[theorem]{Remark}

\usepackage{xspace}

\makeatletter
\def\abbrdot{\@ifnextchar.{}{.\@\xspace}}
\makeatother

\newcommand{\eg}{e.g\abbrdot}

\newcommand{\inftyone}{\ensuremath{(\infty, 1)}}

\DeclareTextFontCommand{\textbfup}{\rmfamily\bfseries\upshape}
\newcommand{\strong}[1]{\textbfup{#1}}

\newcommand*{\blank}{\mathord{-}}

\newcommand*{\lr}[3]{\mathopen{}\left#1{#2}\right#3\mathclose{}}
\newcommand*{\lmr}[5]{\mathopen{}\left#1{#2}\,\middle#3\,{#4}\right#5\mathclose{}}
\newcommand*{\parens}[1]{\lr({#1})}
\newcommand*{\bracket}[1]{\lr[{#1}]}
\newcommand*{\rn}[1]{\lr|{#1}|}
\newcommand*{\set}[1]{\lr\{{#1}\}}
\newcommand*{\setbuilder}[2]{\lmr\{{#1}|{#2}\}}

\newcommand*{\Top}{\textbfup{Top}}
\newcommand*{\Hom}{\textup{Hom}}
\newcommand*{\Map}{\textup{Map}}
\newcommand*{\Ho}[1]{\operatorname{Ho} \parens{#1}}

\newcommand*{\id}{\operatorname{id}}

\title{The classical homotopy category is cartesian closed}
\author{Zhen~Lin Low}

\begin{document}
\begin{abstract}
The category of all topological spaces and homotopy classes of maps is cartesian closed.
In fact, the \inftyone-category of all topological spaces and continuous maps, localised at homotopy equivalences, is cartesian closed.  
We model the exponential object using a topological space whose points are continuous maps equipped with putative neighbourhood filters.
In general, this topological space will not even be $T_1$-separated, but the evaluation map induces a natural weak homotopy equivalence witnessing the universal property.
\end{abstract}

\maketitle

\section*{Introduction}

Let $\Top$ be the category of \emph{all} topological spaces and continuous maps.
(We do \emph{not} restrict to a convenient subcategory.)
It is well known that $\Top$ is \emph{not} cartesian closed: see \eg \cite[Theorem~5.3]{EH:2002}.
By probing with the point $1 = \set{\ast}$, we see that any exponential object in $\Top$ must be (up to isomorphism) the set of continuous maps equipped with an appropriate topology, but this turns out to be impossible in some cases.

On the other hand, $\Top$ is known to be \emph{weakly} cartesian closed: given topological spaces $T$ and $Y$, there exist a topological space $E$ and a continuous map $\epsilon : E \times T \to Y$ such that, for all topological spaces $X$ and all continuous maps $f : X \times T \to Y$, there exists a continuous map $g : X \to E$ such that $f = \epsilon \circ \parens{g \times {\id_T}}$.
In other words, writing $\Hom$ for the set of continuous maps, the following is a surjective map:
\[
\epsilon \circ \parens{\blank \times \id_T} : \Hom \parens{X, E} \to \Hom \parens{X \times T, Y}
\]
See \cite{Rosicky:1999}, particularly the proof of Theorem~2.

Write $\rn{K}$ for the geometric realisation of a simplicial set $K$, and write $\Delta^n$ for the standard $n$-simplex.
Given topological spaces $U$ and $V$, we define the simplicial set $\Map \parens{U, V}$ by the following formula:
\[
\Map \parens{U, V}_n = \Hom \parens{\rn{\Delta^n} \times U, V}
\]
Almost by construction, the set of path components is the set of homotopy classes of continuous maps $U \to V$.
In fact, this is a model for the derived hom space in the Strøm model structure on $\Top$.
What we will see is that, for a particular construction of $E$ and $\epsilon$, we obtain a weak homotopy equivalence of simplicial sets:
\[
\epsilon \circ \parens{\blank \times \id_T} : \Map \parens{X, E} \to \Map \parens{X \times T, Y}
\]
In other words, such $E$ and $\epsilon$ constitute an exponential object in the \inftyone-categorical sense.
It follows that the classical homotopy category $\Ho{\Top}$ is also cartesian closed.

The crucial point is this: in allowing \emph{all} topological spaces while also only demanding a homotopy-coherent universal property, the candidate exponential object $E$ need not satisfy any separation property nor any strict universal property.
Indeed, the $E$ we construct fails to be even $T_1$-separated in general.
Instead of having a single point for each continuous map $T \to Y$ and struggling to construct the perfect function space topology, we smear things out by attaching to each continuous map \emph{all} compatible local notions of convergence in $\Hom \parens{T, Y}$.
In some sense, the resulting $E$ combines all possible spaces of continuous functions $T \to Y$, so we call it the omnibus function space.

To be more precise, recall that a \strong{filter} of a partially ordered set $A$ is an upward-closed subset $\Phi \subseteq A$ such that every finite subset of $\Phi$ has a lower bound in $\Phi$.
(In particular, the empty subset must have a lower bound in $\Phi$, so $\Phi$ is non-empty.)
For example, for any $a \in A$, ${\uparrow} \set{a} = \setbuilder{a' \in A}{a \le a'}$ is a filter of $A$.
The set $\mathscr{F} \parens{A}$ of all filters of $A$ has a natural topology \cite[\S 2]{Scott:1972} in which a subset $W \subseteq \mathscr{F} \parens{A}$ is open if and only if:
\[
\forall \Phi \in W \mathpunct{.} \exists a \in \Phi \mathpunct{.} \forall \Psi \in \mathscr{F} \parens{A} \mathpunct{.} a \in \Psi \Rightarrow \Psi \in W
\]
This is called the \strong{Scott topology} on $\mathscr{F} \parens{A}$.
Note that every Scott-open subset of $\mathscr{F} \parens{A}$ is a union of open subsets of the form ${\uparrow} \set{{\uparrow} \set{a}} = \setbuilder{\Phi' \in \mathscr{F} \parens{A}}{a \in \Phi'}$ where $a \in A$.

In this paper, we will only need the case where $A$ is a powerset.
Conventionally, we speak of filters on a set $S$ instead of filters of the powerset $\mathscr{P} \parens{S}$.
For example, if $S$ is a topological space and $s \in S$, then the neighbourhood filter $\mathscr{N}_s = \setbuilder{Z \subseteq S}{s \in Z^\circ}$ is indeed a filter on (the set of points of) $S$.

Consider the topological space $\hat{S}$ where the points are the $\parens{s, \Phi} \in S \times \mathscr{F} \parens{\mathscr{P} \parens{S}}$ such that $\Phi \subseteq {\uparrow} \set{\set{s}}$, with the initial topology induced by the projection map $\hat{S} \to \mathscr{F} \parens{\mathscr{P} \parens{S}}$.
The map $s \mapsto \parens{s, \mathscr{N}_s}$ is clearly injective; in fact, it is a topological embedding $S \to \hat{S}$.
Varying the topology on $S$, we obtain topological embeddings of every topologisation of $S$ in $\hat{S}$.
Furthermore, every point of $\hat{S}$ is in the image of such a topological embedding.
The omnibus function space is obtained as a certain subspace of $\hat{S}$ when $S = \Hom \parens{T, Y}$.

\subsection*{Acknowledgements and disclaimer}

The arguments in this paper were originated and developed by GPT-6 Luna supervised by GPT-6.1 Sol, then redrafted by GPT-6 Astra.
In particular, the construction (but not the name) of the omnibus function space is due to them.
The named author independently verified the mathematical content and substantially rewrote the manuscript.

This research was undertaken independently of the named author's employment and without the use of employer resources.

\section{The omnibus function space}

Let $T$ and $Y$ be topological spaces and write $\Hom \parens{T, Y}$ for the set of continuous maps $T \to Y$.
Recall that a filter $\Phi$ on $\Hom \parens{T, Y}$ \strong{converges continuously} to $f \in \Hom \parens{T, Y}$ just if, for all $t \in T$ and all open $V \subseteq Y$ with $f \parens{t} \in V$, there exists open $U \subseteq T$ with $t \in U$ such that:
\[
S \parens{U, V} = \setbuilder{g \in \Hom \parens{T, Y} }{ U \subseteq g^{-1} V } \in \Phi
\]
For example, the filter ${\uparrow} \set{\set{f}}$ converges continuously to $f$.

The \strong{omnibus function space} $E \parens{T, Y}$ is the topological space where the points are the $\parens{f, \Phi} \in \Hom \parens{T, Y} \times \mathscr{F} \parens{\mathscr{P} \parens{\Hom \parens{T, Y}}}$ such that $\Phi \subseteq {\uparrow} \set{\set{f}}$ and $\Phi$ converges continuously to $f$, with the initial topology induced by the projection to $\mathscr{F} \parens{\mathscr{P} \parens{\Hom \parens{T, Y}}}$ with the Scott topology.
In other words, $E \parens{T, Y}$ has the topology generated by the basis comprising the sets of the form below,
\[
O_S = \setbuilder{\parens{f, \Phi} \in E \parens{T, Y}}{S \in \Phi}
\]
where $S \subseteq \Hom \parens{T, Y}$.

\begin{lemma}
\label{lem:evaluation}
The evaluation map $\epsilon : E \parens{T, Y} \times T \to Y$,
\[
\epsilon \parens{\parens{f, \Phi}, t} = f \parens{t}
\]
is continuous.
\end{lemma}
\begin{proof}
Let $\parens{f, \Phi} \in E \parens{T, Y}$, let $t \in T$, and let $V \subseteq Y$ be an open neighbourhood of $f \parens{t}$.
By definition, we have an open neighbourhood $U \subseteq T$ with $t \in U$ and $S \parens{U, V} \in \Phi$.
If $\parens{g, \Psi} \in E \parens{T, Y}$ and $S \parens{U, V} \in \Psi$, then $g \in S \parens{U, V}$, so $U \subseteq g^{-1} V$.
Hence, $O_{S \parens{U, V}} \times U$ is an open neighbourhood of $\parens{\parens{f, \Phi}, t} \in E \parens{T, Y} \times T$ contained in $\epsilon^{-1} V$, as required.
\end{proof}

Note that the construction is functorial in $Y$.
Given a continuous map $\theta : Y \to Y'$ and $\parens{f, \Phi} \in E \parens{T, Y}$, define:
\[
\Phi' = \setbuilder{S' \subseteq \Hom \parens{T, Y'}}{\setbuilder{f \in \Hom \parens{T, Y}}{\theta \circ f \in S'} \in \Phi}
\]
It is straightforward to verify that $\parens{\theta \circ f, \Phi'} \in E \parens{T, Y'}$, and also that $\parens{f, \Phi} \mapsto \parens{\theta \circ f, \Phi'}$ is a continuous map $E \parens{T, Y} \to E \parens{T, Y'}$; we omit the details.

Next, we verify the weak exponential property.
Let $X$ be a topological space and let $h : X \to \Hom \parens{T, Y}$ be a map.
Write $f_x$ for the continuous map $h \parens{x} : T \to Y$ and suppose $\parens{x, t} \mapsto f_x \parens{t}$ is a continuous map $e : X \times T \to Y$.
Let $\Phi_x$ be the set of all $S \subseteq \Hom \parens{T, Y}$ such that $x$ is in the interior of $h^{-1} S$.
The \strong{canonical lift} of $h$ is defined by $c \parens{x} = \parens{f_x, \Phi_x}$.

\begin{lemma}
\label{lem:canonical-lift}
For all $x \in X$, $c \parens{x} \in E \parens{T, Y}$.
For all $S \subseteq \Hom \parens{T, Y}$, $c^{-1} O_S$ is the interior of $h^{-1} S$.
In particular, the canonical lift is a continuous map $c : X \to E \parens{T, Y}$.
\end{lemma}
\begin{proof}
It is clear that each $\Phi_x$ is a filter and $\Phi_x \subseteq {\uparrow} \set{\set{f_x}}$.
If $V \subseteq Y$ is an open neighbourhood of $f_x \parens{t}$, then by continuity, we have open $N \subseteq X$ and $U \subseteq T$ such that $x \in N$, $t \in U$, and $N \times U \subseteq e^{-1} V$.
It follows that $x \in N \subseteq h^{-1} S \parens{U, V}$, so $S \parens{U, V} \in \Phi_x$.
Hence, $c \parens{x} = \parens{f_x, \Phi_x} \in E \parens{T, Y}$.

Now, let $S \subseteq \Hom \parens{T, Y}$.
By definition, $c \parens{x} \in O_S$ if and only if $S \in \Phi_x$, and $S \in \Phi_x$ if and only if $x$ is in the interior of $h^{-1} S$.
Hence, $c^{-1} O_S$ is the interior of $h^{-1} S$, as claimed.
\end{proof}

In general, there may be many continuous maps $c_0 : X \to E \parens{T, Y}$ such that $\epsilon \parens{c_0 \parens{x}, t} = f_x \parens{t}$,
but they are all homotopic in the fibre over $h$.
More precisely, writing $I$ for the closed unit interval:

\begin{lemma}
\label{lem:lift-unique}
There is a continuous map $\gamma : I \times X \to E \parens{T, Y}$ such that $\gamma \parens{0, x} = c_0 \parens{x}$, $\gamma \parens{1, x} = c \parens{x}$, and for all $i \in I$, $\epsilon \parens{\gamma \parens{i, x}, t} = f_x \parens{t}$.
\end{lemma}
\begin{proof}
Define $\gamma$ as follows:
\[
\gamma \parens{i, x} = \begin{cases}
c_0 \parens{x} & \text{if } i = 0 \\
c \parens{x} & \text{if } i > 0
\end{cases}
\]
We need to verify continuity.

It suffices to show that $\gamma^{-1} O_S$ is open for all $S \subseteq \Hom \parens{T, Y}$.
Let $A_0 = c_0^{-1} \parens{O_S}$ and $A_1 = c^{-1} \parens{O_S}$.
By lemma~\ref{lem:canonical-lift}, $A_1$ is the interior of $h^{-1} S$.
If $x \in A_0$ and $c_0 \parens{x} = \parens{f_x, \Psi_x}$, then $f_x \in S$ and $S \in \Psi_x$.
Hence, $A_0 \subseteq h^{-1} S$, and since $A_0$ is open, it follows that $A_0 \subseteq A_1$.
Thus, $\gamma^{-1} O_S$ is the union of two open rectangles:
\[
\gamma^{-1} O_S = \parens{I \times A_0} \cup \parens{\left( 0, 1 \right] \times A_1}
\]
In particular, $\gamma^{-1} O_S$ is open, as required.
\end{proof}

\begin{theorem}
\label{thm:ordinary}
There is an induced natural bijection:
\[
\epsilon \circ \parens{\blank \times \id_T} : \pi_0 \Map \parens{X, E \parens{T, Y}} \to \pi_0 \Map \parens{X \times T, Y}
\]
In particular, $E \parens{T, Y}$ and $\epsilon$ have the universal property of an exponential object in $\Ho{\Top}$.
\end{theorem}
\begin{proof}
We certainly have a natural map, and lemma~\ref{lem:canonical-lift} implies it is surjective.
Injectivity remains to be shown.

Let $h : I \times X \to \Hom \parens{T, Y}$ be the transpose of a continuous map $I \times X \times T \to Y$.
Define $h_i : X \to \Hom \parens{T, Y}$ by $h_i \parens{x} = h \parens{i, x}$.
Any lift $c : I \times X \to E \parens{T, Y}$ of $h$ is a homotopy between lifts of $h_0$ and $h_1$.
By lemma~\ref{lem:lift-unique}, any two lifts $X \to E \parens{T, Y}$ of the same $X \to \Hom \parens{T, Y}$ are homotopic.
Hence, by concatenating homotopies, we have injectivity, as required.
\end{proof}

\begin{remark}
Even when $T$ and $Y$ are Hausdorff, $E \parens{T, Y}$ can fail to be $T_1$-separated.
For example, take $T$ to be the point and $Y$ to be the real line.
Then, $\parens{0, \mathscr{N}_0}$ and $\parens{0, {\uparrow} \set{\set{0}}}$ are distinct points of $E \parens{T, Y}$, where $\mathscr{N}_0$ is the neighbourhood filter of $0$ and ${\uparrow} \set{\set{0}}$ is the principal filter, but $\parens{0, {\uparrow} \set{\set{0}}}$ is in every neighbourhood of $\parens{0, \mathscr{N}_0}$, because $\mathscr{N}_0 \subseteq  {\uparrow} \set{\set{0}}$.
Hence, $\parens{0, {\uparrow} \set{\set{0}}}$ is not a closed point of $E \parens{T, Y}$.
\end{remark}

\section{The homotopy-coherent universal property}

As mentioned in the introduction, we will take it for granted that the simplicial set $\Map \parens{U, V}$ models the space of morphisms $U \to V$ in the \inftyone-category of all topological spaces.
This can be justified by considering the Strøm model structure \cite{Strom:1972}: $\rn{\Delta^\bullet} \times U$ is Reedy-cofibrant and resolves $U$, so $\Map \parens{U, V}$ is a homotopy function complex \cite[Chapter~17]{Hirschhorn:2003}.
We omit the details.

In what follows, we also require the fact that, for all \emph{finite} simplicial sets $K$, there is a natural isomorphism of simplicial sets:
\[
\Map \parens{U, V}^K \to \Map \parens{\rn{K} \times U, V}
\]
Very briefly, this is justified by the fact that geometric realisation of finite simplicial sets preserves finitary products and the fact that the canonical comparison map
\[
\int^{\bracket{n} : \mathbf{\Delta}} K_n \cdot \parens{\rn{\Delta^n} \times U} \to \rn{K} \times U
\]
is a homeomorphism when $K$ is finite.
The key observation is that, when $K$ is finite, the canonical map
\[
\coprod_{n} K_{n, \textup{nd}} \cdot \rn{\Delta^n} \to \rn{K}
\]
is a continuous map between compact Hausdorff spaces, hence universally closed; it follows that the induced map
\[
\coprod_{n} K_{n, \textup{nd}} \cdot \parens{\rn{\Delta^n} \times U} \to \rn{K} \times U
\]
is closed, surjective, and continuous, hence is a topological quotient.

\begin{lemma}
\label{lem:relative}
For all finite simplicial sets $K$, the morphism
\[
\Map \parens{X, E \parens{T, Y}}^K \to \Map \parens{X \times T, Y}^K
\]
induced by $\epsilon : E \parens{T, Y} \times T \to Y$ has connected strict fibres.
In particular, it is surjective on vertices.
\end{lemma}
\begin{proof}
Applying the natural isomorphism mentioned above, it suffices to show that every continuous map $\rn{K} \times X \times T \to Y$ admits a lift $\rn{K} \times X \to E \parens{T, Y}$, and that any two such lifts are homotopic in the fibre.
But $X$ was arbitrary in lemmas~\ref{lem:canonical-lift} and~\ref{lem:lift-unique}, so we are done.
\end{proof}

We use the following sufficient condition for a morphism of simplicial sets to be a weak homotopy equivalence: $\phi : A \to B$ is a weak homotopy equivalence if, for every finite simplicial set $K$, $\phi^K : A^K \to B^K$ has connected strict fibres.
Indeed, given finite simplicial sets $K \subseteq L$ and morphisms $a : K \to A$ and $b : L \to B$ such that the evident square commutes, the hypothesis gives us a lift $\tilde{a} : L \to A$ whose restriction to $K$ is \emph{homotopic} to $a$ over $\phi \circ a$.
Applying this in the case where $\parens{K, L}$ is an arbitrary subdivision of $\parens{\partial \Delta^n, \Delta^n}$, we obtain injectivity on $\pi_{n-1}$ (when $n > 0$) and surjectivity on $\pi_n$, so we have a weak homotopy equivalence as claimed.

\begin{theorem}
\label{thm:mapping-spaces}
There is an induced natural weak homotopy equivalence:
\[
\epsilon \circ \parens{\blank \times \id_T} : \Map \parens{X, E \parens{T, Y}} \to \Map \parens{X \times T, Y}
\]
In particular, $E \parens{T, Y}$ and $\epsilon$ have the universal property of a homotopy-coherent exponential object.
\end{theorem}
\begin{proof}
We certainly have a natural simplicial map, and as argued above, lemma~\ref{lem:relative} supplies a sufficient condition for it to be a weak homotopy equivalence.
\end{proof}

\printbibliography
\end{document}